\documentclass[11pt]{article}
\usepackage{amssymb}

\usepackage{amsmath}

\newtheorem{theorem}{Theorem}[section]

\newtheorem{conjecture}[theorem]{Conjectures}

\newtheorem{definition}[theorem]{Definition}
\newtheorem{example}[theorem]{Example}

\newtheorem{lemma}[theorem]{Lemma}

\newtheorem{proposition}[theorem]{Proposition}

\newenvironment{proof}[1][Proof]{\textbf{#1.} }{\ \rule{0.5em}{0.5em}}

\begin{document}

\title{Envelopes of commutative rings}

\author{Rafael Parra\\
Departamento de C\'{a}lculo. Escuela B\'{a}sica \\ Facultad de Ingenier\'{i}a \\ Universidad de los Andes \\  Aptdo. 5101,  M\'{e}rida\\ VENEZUELA\\
{\it rafaelparra@ula.ve}\\ \\ Manuel Saor\'{i}n\\ Departamento de Matem\'{a}ticas\\
Universidad de Murcia, Aptdo. 4021\\
30100 Espinardo, Murcia\\
SPAIN\\ {\it msaorinc@um.es} }

\date{}

\maketitle{}

\begin{abstract}
Given a significative class $\mathcal{F}$ of commutative rings, we
study the precise conditions under which a commutative ring $R$ has
an $\mathcal{F}$-envelope. A full answer is obtained when
$\mathcal{F}$ is the class of fields, semisimple commutative rings
or integral domains. When $\mathcal{F}$ is the class of Noetherian
rings, we give a full answer when the Krull dimension of $R$ is zero
and when the envelope is required to be epimorphic. The general
problem is reduced to identifying the class of non-Noetherian rings
having a monomorphic Noetherian envelope, which we conjecture is the
empty class.
\end{abstract}

\vspace*{0.3cm}


\vspace*{0.3cm}


\vspace*{0.3cm}

{\bf Abbreviated title:} Envelopes of commutative rings

\newpage
\section{Introduction}

The classical concepts of injective envelope and projective cover of
a module led to the introduction of envelopes and covers with
respect to an arbitrary class of objects in a given category. These
more general concepts were introduced by Enochs (\cite{Enochs} and,
under the names of left and right minimal approximations, by
Auslander's school  (\cite{AS-80}, \cite{AR-91}). Lets recall their
definition. Given an arbitrary category $\mathcal{C}$,  a morphism
$f:X\longrightarrow Y$ in it is called left {\bf  minimal} if every
endomorphism $v:Y\longrightarrow Y$ such that $vf=f$ is necessarily
an isomorphism. Dually one defines the concept of  right minimal
morphism.  If  $\mathcal{F}$ is a given class of objects in
$\mathcal{C}$, a $\mathcal{F}$-{\bf preenvelope} of an object $X$ is
a morphism $f:X\longrightarrow F$, with $F\in\mathcal{F}$,
satisfying the property that any morphism $X\longrightarrow F'$ to
an object of $\mathcal{F}$ factors through $f$. When, in addition,
the morphism $f$ is left minimal that preenvelope is called a
$\mathcal{F}$-{\bf envelope}. The concepts of $\mathcal{F}$-precover
and $\mathcal{F}$-cover are defined dually. Since in this paper we
shall
 deal only with preenvelopes and envelopes, a left minimal morphism
 will be
called simply minimal. The study of  envelopes and covers is
generally rather fruitful when the class $\mathcal{F}$ is a
significative one, i.e., a class of objects having nice properties
from different points of views (homological, arithmetical, etc).

In particular, the concepts have proved very useful in module
categories and, more generally, in the context of arbitrary additive
categories. For example,  a long standing open question asked wether
every module had a flat cover. The question was answered
affirmatively by Bican, El Bashir and Enochs (\cite{B-ElB-E}) for
modules over an associative ring with unit and it has turned out to
be very useful in the study of adjoints in the homotopy category of
an abelian category, in particular in the homotopy category of a
module category or that of quasicoherent sheaves on a scheme (see,
e.g. \cite{Neeman} and \cite{Murfet}).

It seems however that, apart from the additive 'world', the concepts
have been somehow neglected. In this paper we consider initially the
situation when $\mathcal{C}=Rings$ is the category of rings (always
associative with unit in this paper) and $\mathcal{F}$ is a
significative class of commutative rings. If $CRings$ denotes the
category of commutative rings, then the forgetful functor
$j:CRings\longrightarrow Rings$ has a left adjoint which associates
to any ring $R$ its quotient $R_{com}$ by the ideal generated by all
differences $ab-ba$, with $a,b\in R$. As a consequence if
$f:R\longrightarrow F$ is a morphism, with $F\in\mathcal{F}$, it is
uniquely factored in the form

\begin{center}
$R\stackrel{pr}{\twoheadrightarrow}R_{com}\stackrel{\bar{f}}{\longrightarrow}F$
\end{center}
and one readily sees that $f$ is a $\mathcal{F}$-(pre)envelope if,
and only if, so is  $\bar{f}$.   That allows us to restrict to the
world of commutative rings all through the paper. So in the sequel,
unless otherwise specified, the term \underline{'ring' will mean}
 \underline{'commutative ring'}.

Our initial motivation for the paper was of geometric nature.
Algebraic schemes have the nicest properties when they enjoy some
sort of Noetherian condition. Therefore it is natural to try to
approximate any given scheme by a Noetherian one and, as usual in
Algebraic Geometry, the first step should be to understand the
affine case. Given the duality between the categories of affine
schemes and  rings \cite{Hartshorne},  our initial task was to
understand envelopes and cover in $CRings$ with respect to the class
of Noetherian rings. But once arrived at this step, it was harmless
to try an analogous study with respect to other significative
classes of  rings (e.g. fields, semisimple rings or domains).

The content of our paper is devoted to the study of envelopes of
rings with respect to those significative classes. The organization
of the papers goes as follows. The results of section 2 are
summarized in the following table, where $\mathcal{F}$ is a class of
commutative rings:

\vspace*{0.3cm}

\begin{tabular}{|c|c|c|}
\hline

 $\mathcal{F}$ & rings R having a $\mathcal{F}$-envelope & $\mathcal{F}$-envelope\\

\hline

fields & R local and $K-dim(R)=0$ & $R\stackrel{pr}{\twoheadrightarrow}R/\mathbf{m}$,  $\mathbf{m}$ maximal\\

 semisimple rings & $Spec(R)$ finite &
$R\stackrel{can}{\longrightarrow}\prod_{\mathbf{p}\in
Spec(R)}k(\mathbf{p})$ \\
integral domains & $Nil(R)$ is a prime ideal &
$R\stackrel{pr}{\twoheadrightarrow}R_{red}=R/Nil(R)$ \\
\hline
\end{tabular}

\vspace*{0.3cm}

In the subsequent sections we study the more complicated case of
{\bf Noetherian envelopes}, i.e., envelopes of rings in  the class
of Noetherian rings. In section 3 we prove that if $R$ is a ring
having a Noetherian preenvelope then $R$ satisfies ACC on radical
ideals and $Spec(R)$ is a Noetherian topological space with the
Zariski topology (Proposition \ref{properties of  a Noetherian
preenvelope}). In section 4 we prove that a ring of zero Krull
dimension has a Noetherian (pre)envelope if, and only it, it is a
finite direct product of local rings which are Artinian modulo the
infinite radical (Theorem \ref{artinian}). In section 5 we show that
a ring $R$ has an epimorphic Noetherian envelope if, and only if, it
has a nil ideal $I$ such that $R/I$ is Noetherian and
$\mathbf{p}I_\mathbf{p}=I_\mathbf{p}$, for all $\mathbf{p}\in
Spec(R)$ (Theorem \ref{epimorphic Noetherian envelopes}). After this
last result, the identification of those rings having a Noetherian
envelope reduces to identify those having a monomorphic Noetherian
envelope. We then  tackle in the final section the problem of the
existence of a non-Noetherian ring with a monomorphic Noetherian
envelope. The existence of such a ring would lead to the existence
of a 'minimal' local one (Proposition \ref{minimal counterexample})
of which the trivial extension $\mathbf{Z}_{(p)}\rtimes\mathbf{Q}$
would be the prototype.  We prove that this ring does not have an
Noetherian envelope (Theorem \ref{nonexistence of Noetherian
envelope}) and conjecture that there does not exist any
non-Noetherian ring having a monomorphic Noetherian envelope.

The notation and terminology on commutative rings followed in the
paper is standard. The reader is referred to  any of the classical
textbooks \cite{AM-69}, \cite{Kunz-85} and \cite {Mat-89}  for all
undefined notions. For the little bit of Category theory that we
need, the reader is referred to \cite{Mitchell}.

\section{Envelopes  of rings in some significative classes}
In this section we will have a class $\mathcal{F}$ of (always
commutative) rings, made precise at each step, and we shall identify
those rings which have a $\mathcal{F}$-(pre)envelope. A trivial but
useful fact will be used all through, namely, that if
$f:R\longrightarrow F$ is a $\mathcal{F}$-(pre)envelope, then the
inclusion $Im(f)\hookrightarrow F$ is also a
$\mathcal{F}$-(pre)envelope.

Our first choice of $\mathcal{F}$ is the class of fields or the
class of semisimple rings. For the study of envelopes in these
classes,  the following well-known result will be used. We include a
short proof for completeness.

\begin{lemma} \label{no morphisms between residue fields}
Let $\mathbf{p}$ and $\mathbf{q}$ two prime ideals of $R$ and
$u_\mathbf{p}:R\longrightarrow k(\mathbf{p})$ and
$u_\mathbf{q}:R\longrightarrow k(\mathbf{q})$ the canonical ring
homomorphisms to the respective residue fields. If
$h:k(\mathbf{p})\longrightarrow k(\mathbf{q})$ is a field
homomorphism such that $hu_\mathbf{p}=u_\mathbf{q}$, then
$\mathbf{p}=\mathbf{q}$ and $h=1_{k(\mathbf{p})}$ is the identity
map.
\end{lemma}
\begin{proof}
We have $u_\mathbf{q}(\mathbf{p})=(hu_\mathbf{p})(\mathbf{p})=0$.
That means that the ideal $\mathbf{p}R_\mathbf{q}$ of $R_\mathbf{q}$
is mapped onto zero by the canonical projection
$R_\mathbf{q}\twoheadrightarrow k(\mathbf{q})$. In case
$\mathbf{p}\not\subseteq\mathbf{q}$, that leads to contradiction for
$\mathbf{p}R_\mathbf{q}=R_\mathbf{q}$. So we can assume
$\mathbf{p}\subseteq\mathbf{q}$. If this inclusion is strict, then
we  choose an element $s\in\mathbf{q}\setminus\mathbf{p}$ and get
that
$h(s+\mathbf{p})=h(u_\mathbf{p}(s))=u_\mathbf{q}(s)=s+\mathbf{q}=0$.
Since every field homomoprhism  is injective, we conclude that
$s+\mathbf{p}=0$ in $k(\mathbf{p})$, which is false because $0\neq
s+\mathbf{p}\in R/\mathbf{p}\subset k(\mathbf{p})$.

We then necessarily have $\mathbf{p}=\mathbf{q}$. If we denote by
$i_\mathbf{p}:R/\mathbf{p}\hookrightarrow k(\mathbf{p})$ the
inclusion, then we get that $hi_\mathbf{p}=i_\mathbf{p}$ and, by the
universal property of localization with respect to multiplicative
sets, we conclude that $h=1_{k(\mathbf{p})}$.
\end{proof}

\begin{theorem} \label{field and semisimple envelopes}
Let $\mathcal{F}$ be the class of fields, let $\mathcal{S}$ be the
class of semisimple commutative rings and let $R$ be any given
commutative ring. The following assertions hold:

\begin{enumerate}
\item $R$ has a $\mathcal{F}$-(pre)envelope if, and only if, $R$ is
local and $K-dim(R)=0$. In that case, the projection
$R\twoheadrightarrow R/\mathbf{m}$ is the $\mathcal{F}$-envelope,
where $\mathbf{m}$ is the maximal ideal. \item $R$ has a
$\mathcal{S}$-(pre)envelope if, and only if, $Spec(R)$ is finite. In
that case, the canonical map $R\longrightarrow\prod_{\mathbf{p}\in
Spec(R)}k(\mathbf{p})$ is the $\mathcal{S}$-envelope.
\end{enumerate}
\end{theorem}
\begin{proof}
1) Suppose that $f:R\longrightarrow F$ is a
$\mathcal{F}$-preenvelope. If $\mathbf{m}$ is any maximal ideal of
$R$, then the canonical projection $p:R\twoheadrightarrow
R/\mathbf{m}$ factors through $f$, so that we have a field
homomorphism $h:F\longrightarrow R/\mathbf{m}$ such that $hf=p$.
Then $f(\mathbf{m})\subseteq Ker(h)=0$,  so that
$\mathbf{m}\subseteq Ker(f)$ and hence $\mathbf{m}=Ker(f)$. It
follows that $R$ is local with $Ker(f)$ as unique maximal ideal. Let
then $\bar{f}:R/Ker(f)\longrightarrow F$ be the field homomorphism
such that $\bar{f}p=f$ and, using the $\mathcal{F}$-preenveloping
condition of $f$, choose a field homomorphism $g:F\longrightarrow
R/Ker(f)$ such that $gf=p$. Then we have that $g\bar{f}p=gf=p$, so
that $g\bar{f}=1$ and hence $\bar{f}$ and $g$ are isomorphisms.
Since there is a canonical ring homomorphism $R\longrightarrow
k(\mathbf{p})$ for every $\mathbf{p}\in Spec(R)$, Lemma \ref{no
morphisms between residue fields} implies that $Spec(R)=\{Ker(f)\}$
and, hence, that $K-dim(R)=0$.

Conversely, let $R$ be a local ring with maximal ideal $\mathbf{m}$
such that $K-dim(R)=0$. Any ring homomorphism $f:R\longrightarrow
F$, which $F$ field, has a prime ideal as kernel. Then
$Ker(f)=\mathbf{m}$ and $f$ factors through the projection
$p:R\twoheadrightarrow R/\mathbf{m}$. This projection is then the
$\mathcal{F}$-envelope of $R$.

2) It is well-known that a commutative ring is semisimple if, and
only if, it is a finite direct product of fields. Given a ring
homomorphism $f:R\longrightarrow S$, with $S$ semisimple, it follows
that $f$ is a $\mathcal{S}$-preenvelope if, and only if, every ring
homomorphism $g:R\longrightarrow K$, with $K$ a field, factors
through $f$. If we fix a decomposition $S=K_1\times ...\times K_r$,
where the $K_i$ are fields, any ring homomorphism
$h:S\longrightarrow K$ to a field vanishes on all but one of the
canonical idempotents $e_i=(0,...,\stackrel{i}{1},...0)$, so that
$h$ can be represented by a matrix map

\begin{center}
$h=\begin{pmatrix} 0 & ...& 0 & h' & 0 & ...&
0\end{pmatrix}:K_1\times ...\times K_r\longrightarrow K$,
\end{center}
where $h':K_i\longrightarrow K$ is a field homomorphism.  We shall
frequently use these facts.

Suppose that $f=\begin{pmatrix} f_1\\ .\\ .\\
f_r\end{pmatrix}:R\longrightarrow K_1\times ...\times K_r$ is a
$\mathcal{S}$-preenvelope (each $f_i:R\longrightarrow K_i$ being a
ring homomorphism). For every $j\in\{1,...,r\}$, we  put
$\mathbf{p}_j:=Ker(f_j)$, which is a prime ideal of $R$. By the
universal property of localization, there is a unique field
homomorphism $g_j:k(\mathbf{p}_j)\longrightarrow K_j$ such that
$g_ju_{\mathbf{p}_j}=f_j$. Let now $\mathbf{p}\in Spec(R)$ be any
prime ideal. The canonical map $u_\mathbf{p}:R\longrightarrow
k(\mathbf{p})$ factors through $f$ and, by the last paragraph we get
an index $i\in\{1,...,r\}$ together with a morphism
$h':K_i\longrightarrow k(\mathbf{p})$ such that
$h'f_i=u_\mathbf{p}$. But then $h'g_iu_{\mathbf{p}_i}=u_\mathbf{p}$
and
 Lemma \ref{no morphisms between residue
fields} tells us that $\mathbf{p}=\mathbf{p}_i$. That proves that
$Spec(R)=\{\mathbf{p}_1,...,\mathbf{p}_r\}$.

Conversely, suppose that $Spec(R)=\{\mathbf{p}_1,
...,\mathbf{p}_r\}$ is finite. If $g:R\longrightarrow K$ is a ring
homomorphism, with $K$ a field, then $\mathbf{q}:=Ker(g)$ is a prime
ideal and $g$ factors through $u_\mathbf{q}:R\longrightarrow
k(\mathbf{q})$ and, hence, also through the canonical map
$f:R\longrightarrow\prod_{1\leq i\leq r}k(\mathbf{p}_i)$. So $f$
becomes a $\mathcal{S}$-preenvelope. It only remains to check that
it is actually an envelope. Indeed, if $\varphi :\prod_{1\leq i\leq
r}k(\mathbf{p}_i)\longrightarrow\prod_{1\leq i\leq
r}k(\mathbf{p}_i)$ is a ring homomorphism such that $\varphi f=f$
then, bearing in mind that $f=\begin{pmatrix} u_{\mathbf{p}_1}\\
.\\ .\\ u_{\mathbf{p}_r}\end{pmatrix}$, Lemma \ref{no morphisms
between residue fields} tells us that $\varphi =1_{\prod
k(\mathbf{p}_i)}$ is the identity map.
\end{proof}

\begin{example}
A Noetherian ring of zero Krull dimension is a typical example of
ring having a semisimple envelope. An example with nonzero Krull
dimension  is given by a discrete valuation domain $D$ (e.g. the
power series algebra $K[[X]]$ over the field $K$).
\end{example}

We end this section with the characterization of rings which have
preenvelopes in the class of integral domains. Recall that a ring
$R$ is reduced if $Nil\left( R\right) =0$. We define the reduced
ring associated to a ring $R$ as $R_{red}=R/Nil\left( R\right) $.

\begin{proposition} \label{envelope in the class of integral
domains} Let $R$ be a commutative ring and $\mathcal{D}$ the class
of integral domains. The following conditions are equivalent:

\begin{enumerate}
\item $R$ has a $\mathcal{D}$-(pre)envelope;

\item $Nil\left( R\right) $ is a prime ideal of $R$.
\end{enumerate}
In that case,  the projection $p:R\twoheadrightarrow R_{red}$ is the
$\mathcal{D}$-envelope.
\end{proposition}
\begin{proof}
$2)\Longrightarrow 1)$ Every ring homomorphism $f:R\longrightarrow
D$, with $D\in\mathcal{D}$, vanishes on $Nil(R)$. That proves that
$f$ factors through $p:R\twoheadrightarrow R_{red}$, so that this
latter map is a $\mathcal{D}$-envelope.

$1)\Longrightarrow 2)$ Let $f:R\longrightarrow D$ be a
$\mathcal{D}$-preenvelope. Then $f(Nil(R))=0$ and the induced map
$\bar{f}:R_{red}\longrightarrow D$ is also a
$\mathcal{D}$-preenvelope. Replacing $R$ by $R_{red}$ if necessary,
we can and shall assume that $R$ is reduced and will have to prove
that then $R$ is an integral domain.

Indeed $\mathbf{q}:=Ker(f)$ is a prime ideal and the projection
$\pi_\mathbf{q}:R\twoheadrightarrow R/\mathbf{q}$ is also a
$\mathcal{D}$-preenvelope. It is actually a $\mathcal{D}$-envelope
since it is surjective. But then, for every $\mathbf{p}\in Spec(R)$,
the projection $\pi_\mathbf{p}:R\twoheadrightarrow R/\mathbf{p}$
factors through $\pi_\mathbf{q}$. This implies that
$\mathbf{q}\subseteq\mathbf{p}$, for every $\mathbf{p}\in Spec(R)$,
and hence that $\mathbf{q}\subset\bigcap_{\mathbf{p}\in
Spec(R)}\mathbf{p}=Nil(R)=0$ (cf. \cite{Kunz-85}[Corollary I.4.5]).
Therefore $0=\mathbf{q}$ is a prime ideal, so that $R$ is an
integral domain.
\end{proof}

\section{Rings with a Noetherian preenvelope}

 All through this section we fix a ring $R$ having a Noetherian preenvelope  $f:R\longrightarrow N$.
  An ideal $I$ of $R$ will
be called {\bf restricted} if $I=f^{-1}(f(I)N)$ or, equivalently, if
$I=f^{-1}(J)$ for some ideal $J$ of $N$. The following result
gathers some useful properties of the rings having a Noetherian
preenvelope.

\begin{proposition} \label{properties of  a Noetherian preenvelope}

Let $f:R\longrightarrow N$ be a Noetherian preenvelope. The
following assertions hold:

\begin{enumerate}
\item $Ker(f)$ is contained in $Nil(R)$
\item Every radical ideal of $R$ is restricted \item $R$ satisfies
ACC on restricted ideals \item $Spec(R)$ is a Noetherian topological
space with the Zariski topoloy. In particular, if $I$ is an ideal of
$R$ there are only finitely many prime ideals minimal over $I$
\end{enumerate}
\end{proposition}
\begin{proof}
1) If follows from the fact that, for every $\mathbf{p}\in Spec(R)$,
the canonical map $u_\mathbf{p}:R\longrightarrow k(\mathbf{p})$
factors through $f$, and hence $Ker(f)\subseteq
Ker(u_\mathbf{p})=\mathbf{p}$.

2) Since the assignment $J\rightsquigarrow f^{-1}(J)$ preserves
intersections, it will be enough to prove that every prime ideal of
$R$ is restricted. If $g:R\longrightarrow A$ is any ring
homomorphism and we consider the $R$-module structures on $A$ and
$N$   given by restriction of scalars via $g$ and $f$, respectively,
then $A\otimes_RN$ becomes an $R$-algebra which fits in the
following pushout in $CRings$:

\vspace*{0.3cm}

\setlength{\unitlength}{1mm}
\begin{picture}(140,22)

 \put(40,20){$R$} \put(45,21){\vector(1,0){14}} \put(62,20){$N$}
\put(42,18){\vector(0,-1){12}} \put(64,18){\vector(0,-1){12}}

\put(40,2){$A$} \put(45,3){\vector(1,0){12}}
\put(60,2){$A\otimes_RN$}

\put(50,22){$f$}

\put(37,12){$g$}

\put(56,6){$\bigstar$}

\end{picture}

In case $A\otimes_RN$ is Noetherian, the universal property of
pushouts tells us that the botton map $A\longrightarrow A\otimes_RN$
is a Noetherian preenvelope of $A$. We shall frequently use this
fact in the paper.

For our purposes in this proof, we take $A=R/\mathbf{p}$ and
$g=\pi_\mathbf{p}:R\twoheadrightarrow R/\mathbf{p}$, the projection,
for any fixed $\mathbf{p}\in Spec(R)$.  Then the map
$\bar{f}:R/\mathbf{p}\longrightarrow N/f(\mathbf{p})N\cong
(R/\mathbf{p})\otimes_RN$ is a Noetherian preenvelope and, in
particular,  the inclusion $i_\mathbf{p}:R/\mathbf{p}\hookrightarrow
k(\mathbf{p})$ factors through it. This implies that
$f^{-1}(f(\mathbf{p})N)/\mathbf{p}=Ker(\bar{f})\subseteq
Ker(i_\mathbf{p})=0$ and, hence, that $\mathbf{p}$ is a restricted
ideal.

3) Clear since $N$ is  Noetherian and every ascending chain
$I_0\subseteq I_1\subseteq ...$ of restricted ideals of $R$ is the
preimage of the chain $f(I_0)N\subseteq f(I_1)N\subseteq ...$ of
ideals of $N$.

4) There is an order-reversing bijection between (Zariski-)closed
subsets of $Spec(R)$ and radical ideals of $R$. Therefore $Spec(R)$
is Noetherian if, and only if, $R$ has ACC on radical ideals (cf.
\cite{Kunz-85}[Chapter I, section 2]). But this latter property is
satisfied due to assertions 2) and 3). Finally, if $I$ is any ideal
of $R$ then the prime ideals of $R$ which are minimal over $I$ are
precisely those corresponding to the irreducible components of the
closed subset $\mathcal{V}(I)=\{\mathbf{p}\in Spec(R):$
$I\subseteq\mathbf{p}\}$, which is a Noetherian topological space
since so is $Spec(R)$. Therefore those prime ideals are a finite
number (cf. \cite{Kunz-85}[Proposition I.2.14]).
\end{proof}

\section{The case of Krull dimension zero}

In this section we shall identify the rings of zero Krull dimension
having a Noetherian (pre)envelope.

\begin{lemma} \label{preenvelopes of finite directproducts}
Let $R=R_1\times ...\times R_n$ be a ring decomposed into a finite
product of nonzero rings. The following assertions are equivalent:

\begin{enumerate}
\item $R$ has a Noetherian (pre)envelope \item Each $R_i$ has a
Noetherian (pre)envelope.
\end{enumerate}
In such a case, if $f_i:R_i\longrightarrow N_i$ is a Noetherian
(pre)envelope for each $i=1,...,n$, then the diagonal map $f=diag
(f_1,...,f_n):R_1\times ...\times R_n\longrightarrow N_1\times
...\times N_n$ is a Noetherian (pre)envelope.
\end{lemma}
\begin{proof}
We shall prove the equivalence of 1) and 2) for the case of
preenvelopes, leaving the minimality of morphisms for the end.

$1)\Longrightarrow 2)$ Fix a Noetherian preenvelope $f:R=R_{1}\times
\cdots \times
R_{n}\longrightarrow N$. If now $%
e_{i}=\left( 0,\ldots ,\overset{i}{1},0,\ldots ,0\right) $ $(i=1,\ldots ,n)$%
, then we have that $N_{i}=:f\left( e_{i}\right) N$ is a Noetherian
ring, for every $i=1,\ldots ,n$. Clearly $f\left( R_{i}\right)
\subseteq N_{i}$ and we have an induced ring homomorphism
$f_{i}=f\shortmid _{R_{i}}:R_{i}\longrightarrow N_{i}$, for every
$i=1,\ldots ,n.$. We clearly have $N\cong N_1\times ...\times N_r$
and $f$ can be identified with the diagonal map

\begin{center}
$diag (f_1,...,f_n):R_1\times ...\times R_n\longrightarrow N_1\times
...\times N_n$.
\end{center}

Let \ now $%
h:R_{i}\longrightarrow N^{\prime }$ be a ring homomorphism, with
$N^{\prime } $ Noetherian. Then the matrix map

\begin{center}
$\begin{pmatrix}0 & ..& 0 & \stackrel{i}{h} & 0 & ..& 0
\end{pmatrix}:R_1\times ...\times R_n\longrightarrow N^{\prime }$
\end{center}
is also a ring homomorphism, which must factor through $f\equiv
diag(f_1,...,f_n)$. That implies that $h$ factors through $f_i$, so
that $f_i$ is a Noetherian preenvelope for each $i=1,...,n$.

$2)\Longrightarrow 1)$ Suppose that \ $f_{i}:R_{i}\longrightarrow
N_{i}$ is a Noetherian preenvelope for $i=1,\ldots ,n$ and let us
put

\begin{center}
$f:=diag(f_1,...,f_n):R_1\times ...\times R_n\longrightarrow
N_1\times ...\times N_n$.
\end{center}
Given any ring homomorphism $g:R=$ $R_{1}\times \cdots \times
R_{n}\longrightarrow N^{\prime }$, where $N^{\prime }$ is a
Noetherian ring, put $N'_i=N'g(e_i)$ for $i=1,...,n$. Then we have
an isomorphism $N'\cong N'_1\times ...\times N'_n$ and the $N'_i$
are also Noetherian rings, some of them possibly zero. Viewing that
isomorphism as an identification, we can think of $g$ as a diagonal
matrix map

\begin{center}
$g\equiv diag (g_1,...,g_n):R_1\times ...\times R_n\longrightarrow
N'_1\times ...\times N'_n\cong N'$,
\end{center}
where each $g_i$ is a ring homomorphism. Then $g_i$ factors through
$f_i$, for every $i=1,...,n$, and so $g$ factors through $f$.
Therefore $f$ is a Noetherian preenvelope.

We come now to the minimality of morphisms. We can consider a
Noetherian preenvelope given by a diagonal map
$f=diag(f_1,...,f_n):R_1\times ...\times R_n\longrightarrow
N_1\times ...\times N_n$, where each $f_i$ is a Noetherian
preenvelope. If $f$ is a minimal morphism in $CRings$ one readily
sees that each $f_i$ is also minimal. Conversely, suppose that each
$f_i$ is minimal and consider any ring homomorphism $\varphi
:N_1\times ...\times N_n\longrightarrow N_1\times ...\times N_n$
such that $\varphi f=f$. We can identify $\varphi$ with a matrix
$(\varphi_{ij})$, where $\varphi _{ij}:N_j\longrightarrow N_i$ is a
a map preserving addition and multiplication (but not necessarily
the unit) for all $i,j\in\{1,...,n\}$. Viewing the equality $\varphi
f=f$ as a matricial equality, we get:

\begin{center}
$\varphi_{ij}f_j=0$, for $i\neq j$

$\varphi_{ii}f_i=f_i$.
\end{center}
The first equality for $i\neq j$ gives that
$\varphi_{ij}(1)=(\varphi_{ij}f_j)(1)=0$,  and then $\varphi_{ij}=0$
since $\varphi_{ij}$ preserves multiplication. Therefore $\varphi
=diag (\varphi_{11},...,\varphi_{nn})$ and the minimality of the
morphisms $f_i$ gives that each $\varphi_{ii}$ is a ring
isomorphism. It follows that $\varphi$ is an isomorphism.
\end{proof}

\begin{lemma} \label{structure of zero Krull dimension}
If $R$ has a Noetherian preenvelope and  $K-dim(R)=0$ then $R$ is
finite direct product of local rings with zero Krull dimension.
\end{lemma}
\begin{proof}
If $K-dim(R)=0$ and $R$ has a Noetherian preenvelope then, since all
its prime ideals are both maximal and minimal,  Proposition
\ref{properties of  a Noetherian preenvelope} tells us that there
are only finitely many of them. Then $R$ is a semilocal ring and,
since $Nil(R)$ is a nil ideal, idempotents lift modulo $Nil(R)$ (cf.
\cite{Ste-75}[Proposition VIII.4.2]). The result then follows
immediatly since $R_{red}=R/Nil(R)$ is a finite direct product of
fields (cf. \cite{Kunz-85}, Proposition I.1.5).
\end{proof}

The last two lemmas reduce our problem to the case of a local ring.
We start by considering the case in which $R$ has a monomorphic
Noetherian preenvelope.

\begin{lemma}
\label{nilpotent}Let $R$ be a local ring with maximal ideal
$\mathbf{m}$ and $K$-$dim\left( R\right) =0$. If $R$ has a
monomorphic Noetherian preenvelope then $\mathbf{m}$ is nilpotent.
In particular, the following conditions are equivalent:

\begin{enumerate}
\item $R$ is a Noetherian ring;

\item $R$ is an Artinian ring;

\item $\mathbf{m}/\mathbf{m}^2$ is a finite dimensional $R/\mathbf{m}$-vector space.
\end{enumerate}
\end{lemma}

\begin{proof}
Let $j:R\longrightarrow N$ be a Noetherian monomorphic preenvelope
of $R$. Since $\mathbf{m}=Nil(R)$ is nil it follows that
$\mathbf{m}N$ is a nil ideal in a Noetherian ring. Then it is
nilpotent, and so there is a $n>0$ such that $\mathbf{m}^n\subset
(\mathbf{m}N)^n=0$.

 For
the second part, note that when $K$-$dim\left( R\right) =0$ then $R$
is Artinian if and only if $R$ is Noetherian (see \cite[Theorem
8.5]{AM-69}). Assume that $\mathbf{m}/\mathbf{m}^2$ is finitely
generated. If $\{x_1,...,x_r\}$ is a finite set of generators of
$\mathbf{m}$ modulo $\mathbf{m}^2$, then the products $x_{\sigma
(1)}\cdot ...\cdot x_{\sigma (m)}$, with $\sigma$ varying in the set
of maps $\{1,...,m\}\longrightarrow\{1,...,r\}$, generate
$\mathbf{m}^m/\mathbf{m}^{m+1}$ both as an $R$-module and as a
$R/\mathbf{m}$-vector space. In particular, each
$\mathbf{m}^m/\mathbf{m}^{m+1}$ ($m=0,1,...$) is an $R$-module of
finite length. Since there is a $n>0$ such that $\mathbf{m}^n=0$ we
conclude that $R$ has finite length as $R$ module, i.e., $R$ is
Artinian.
\end{proof}

\begin{lemma} \label{no intermediate in epimorphic envelope}
Let $A$ be a non-Noetherian commutative ring and $\mathbf{a}$ be an
ideal of $A$ such that the projection $p:A\twoheadrightarrow
A/\mathbf{a}$ is a Noetherian envelope. Then $\mathbf{a}$ does not
have a simple quotient.
\end{lemma}
\begin{proof}
Any simple quotient of $\mathbf{a}$ is isomorphic to
$\mathbf{a}/\mathbf{a}'$, for some ideal $\mathbf{a}'$ such that
$\mathbf{a}'\subsetneq\mathbf{a}$. Now the canonical exact sequence

\begin{center}
$0\rightarrow\mathbf{a}/\mathbf{a}'\hookrightarrow
A/\mathbf{a}'\twoheadrightarrow A/\mathbf{a}\rightarrow 0$
\end{center}
has the property that its outer nonzero terms are Noetherian
$A$-modules. Then  $A/\mathbf{a}'$ is  Noetherian, both as an
$A$-module and as a ring. But then the projection
$q:A\twoheadrightarrow A/\mathbf{a}'$ factors through the Notherian
envelope $p$, which implies that $\mathbf{a}=Ker(p)\subseteq
Ker(q)=\mathbf{a}'$. This contradicts our choice of $\mathbf{a}'$.
\end{proof}

\begin{lemma} \label{auxiliar for zero Krull dimension}
Let $R$ be a local ring with zero Krull dimension having a
monomoprhic Noetherian preenvelope. Then $R$ is Artinian.
\end{lemma}
\begin{proof}
By Lemma \ref{nilpotent} it is enough to prove that $R$ is
Noetherian. Suppose then that there exists a non-Noetherian local
$R$ such that $K-dim(R)=0$ and $R$ has a monomorphic Noetherian
preenvelope. Fix such a preenvelope $j:R\longrightarrow N$ and view
it as an inclusion. The set of restricted ideals $I$ of $R$ such
that $R/I$ is not Noetherian has a maximal element, say $J$ (see
Proposition \ref{properties of  a Noetherian preenvelope}). Note
that the induced map $\bar{j}:R/J\longrightarrow N/JN$ is also a
monomorphic Noetherian preenvelope. Then, replacing $R$ by $R/J$ if
necessary, we can and shall assume that $R/I$ is Noetherian, for
every restricted ideal $I\neq 0$.

By the proof of assertion 2 in Proposition \ref{properties of  a
Noetherian preenvelope}, we know that the induced map
$R/\mathbf{m}^2\longrightarrow N/\mathbf{m}^2N$ is a Noetherian
preenvelope and it factors in the form

\begin{center}
$R/\mathbf{m}^2\stackrel{p}{\twoheadrightarrow}R/I\rightarrowtail
N/\mathbf{m}^2N$,
\end{center}
where $I:=\mathbf{m}^2N\cap R$. If $\mathbf{m^2}\neq 0$ then
$p:R/\mathbf{m}^2\twoheadrightarrow R/I$ would be a Noetherian
envelope because $I$ is nonzero and restricted. Since, by Lemma
\ref{nilpotent}, $R/\mathbf{m}^2$ is not a Noetherian ring it
follows that $\mathbf{m}^2\subsetneq I$. But then we contradict
Lemma \ref{no intermediate in epimorphic envelope} for
$I/\mathbf{m}^2$ is a semisimple $R/\mathbf{m}^2$-module (it is
annihilated by $\mathbf{m}/\mathbf{m}^2$) and, hence, it always has
simple quotients.

We then have that $\mathbf{m}^2=0$. Consider the set

\begin{center}
$\mathcal{I}=\{\mathbf{a}\text{ ideal of }R:$ $0\neq
\mathbf{a}\subseteq\mathbf{m}\text{ and }\mathbf{a}\text{ not
restricted}\}$.
\end{center}
In case $\mathcal{I}\neq\emptyset$, we pick up any
$\mathbf{a}\in\mathcal{I}$. Then $\mathbf{a}N\cap R/\mathbf{a}$ is a
nonzero semisimple $R$-module since it is a subfactor of the
$R/\mathbf{m}$-vector space $\mathbf{m}$. We can then find an
intermediate ideal $\mathbf{a}\subseteq J\subsetneq \mathbf{a}N\cap
R=:J'$ such that $J'/J$ is a simple module. Notice that
$\mathbf{a}N=JN=J'N$. Now the induced map
$\tilde{j}:R/J\longrightarrow N/JN$ is a Noetherian preenvelope and
an argument already used in the previous paragraph shows that the
projection $R/J\twoheadrightarrow R/J'$ is a Noetherian envelope.
This contradicts Lemma \ref{no intermediate in epimorphic envelope}
for $J'/J$ is simple. Therefore we get $\mathcal{I}=\emptyset$.
Since $\mathbf{m}$ is a semisimple $R$-module, we can take a minimal
ideal $I_0\subset\mathbf{m}$, which is then  necessarily restricted.
Then $I_0$ and $R/I_0$ are both Noetherian $R$-modules, from which
we get that $R$ is a Noetherian ring and, hence, a contradiction.
\end{proof}

\begin{definition}
Let $R$ be a local ring with maximal ideal $\mathbf{m}$ such that
$K-dim(R)=0$. We shall call it {\bf Artinian modulo the infinite
radical} in case $R/\bigcap_{n>0}\mathbf{m}^n$ is Artinian.
Equivalently, if there is an integer $n>0$ such that
$\mathbf{m}^n=\mathbf{m}^{n+1}$ and $R/\mathbf{m}^n$ is Artinian.
\end{definition}

We are now ready to prove the main result of this section.

\begin{theorem} \label{artinian}
Let $R$ be a commutative ring such that $K$-$dim\left( R\right) =0$.
The following assertions are equivalent

\begin{enumerate}
\item $R$ has a Noetherian (pre)envelope
\item $R$ is isomorphic to a finite product $R_1\times ...\times R_r$,   where the $R_i$ are  local
rings  which are Artinian modulo the infinite radical.
\end{enumerate}
In that situation, if $\mathbf{m}_i$ is the maximal ideal of $R_i$
and $p_i:R_i\twoheadrightarrow R_i/\bigcap_{n>0}\mathbf{m}_i^n$ is
the projection, for each $i=1,...,r$, then the 'diagonal' map
$diag(p_1,...,p_n):R_1\times ...\times R_r\longrightarrow
R_1/\bigcap_{n>0}\mathbf{m}_1^n\times ...\times
R_r/\bigcap_{n>0}\mathbf{m}_r^n$ is the Noetherian envelope.

\end{theorem}

\begin{proof}
  Using Lemmas \ref{structure of zero
Krull dimension} and \ref{preenvelopes of finite directproducts},
the proof
 reduces to the case when $R$ is local, something that we assume in
the sequel.

$1)\Longrightarrow 2)$ Let $f:R\longrightarrow N$ be the Noetherian
preenvelope. Then we
have a factorization%
\begin{equation*}
f:R\overset{p}{\twoheadrightarrow }R/Ker\left( f\right) \overset{\overline{f}%
}{\rightarrowtail }N\text{,}
\end{equation*}%
where $\overline{f}$ $\ $is a Noetherian (monomorphic) preenvelope.
It follows from Lemma \ref{auxiliar for zero Krull dimension} that
$R/Ker\left( f\right) $ is artinian. Putting $I=Ker\left( f\right) $, we then get that the projection $R\overset{p}{%
\longrightarrow }R/I$ is the Noetherian envelope. The case $I=0$ is
trivial, for then $R$ is Artinian. Suppose that $I\neq 0$. In case
$I^2\neq I$, we get a contradiction with Lemma \ref{no intermediate
in epimorphic envelope} for $I/I^2$ is a nonzero module over the
Artinian ring $R/I$ and, hence, it always has simple quotients.
Therefore we have $I=I^2$ in that case, which implies  that $I=I^n$,
for all $n>0$, and hence  that
$I\subseteq\bigcap_{n>0}\mathbf{m}^n$. But, since $R/I$ is Artinian,
we have that $\mathbf{m}^n\subseteq I$ for $n>>0$. It follows that
there exists a $k>0$ such that $I=\mathbf{m}^n$, for all $n\geq k$.
Then $R$ is Artinian modulo the infinite radical.

 $2)\Longrightarrow 1)$ Let $R$ be local and  Artinian modulo the
 infinite radical and let $n$ be the smallest of the positive integers $k$ such that
 $\mathbf{m}^k=\mathbf{m}^{k+1}$. Then  $R/\mathbf{m}^n=R/\bigcap_{n>0}\mathbf{m}^n$ is Artinian.

 We shall prove that if $h:R\longrightarrow N$ is any ring
 homomorphism, with $N$ Noetherian, then $h(\mathbf{m}^n)=0$, from
 which it will follow that the projection $p:R\twoheadrightarrow
 R/\mathbf{m}^n=R/\bigcap_{k>0}\mathbf{m}^k$ is the Noetherian
 envelope. Since $\mathbf{m}=Nil(R)$ is a nil ideal of $R$ it follows that
 $h(\mathbf{m})N$ is a nil ideal of the Noetherian ring $N$, and thus nilpotent. But the
 equality $\mathbf{m}^n=\mathbf{m}^{k}$ implies that
 $(h(\mathbf{m})N)^n=(h(\mathbf{m})N)^k$, for all $k\geq n$. It then
 follows that $h(\mathbf{m}^n)N=(h(\mathbf{m})N)^n=0$ and, hence,
 that $h(\mathbf{m}^n)=0$.

 The final statement is a direct consequence of the above paragraphs and of Lemma \ref{preenvelopes of finite
 directproducts}.
\end{proof}

\begin{example} \label{Fibonacci}
Let $a_n$ ($n=1,2,...$) by the $n$-th term of the Fibonacci sequence
$1,1,2,3,5,...$ and within the power series algebra
$K[[X_1,X_2,...]]$  over the field $K$, consider the ideal $I$
generated by the following relations:

\begin{center}
$X_n=X_{n+1}X_{n+2}$

$X_n^{a_n+1}=0$,
\end{center}
for all positive integers $n$. Then $R=K[[X_1,X_2,...]]/I$ is a
local ring of zero Krull dimension which is Artinian modulo the
infinite radical, but is not Artinian.
\end{example}
\begin{proof}
Since $\mathbf{m}=(X_1,X_2,...)$ is the only maximal ideal of
$K[[X_1,X_2,...]]$ it follows that $\bar{\mathbf{m}}=\mathbf{m}/I$
is the only maximal ideal of $R$, so that $R$ is local. On the other
hand, the second set of relations tells us that $\bar{\mathbf{m}}$
is a nil ideal for all generators $x_n:=X_n+I$ of $\bar{\mathbf{m}}$
are nilpotent elements. In particular, we have $K-dim(R)=0$.

On the other hand, the first set of relations tells us that
$x_n\in\bar{\mathbf{m}}^2$, for all $n>0$, so that
$\bar{\mathbf{m}}=\bar{\mathbf{m}}^2$. Then
$R/\bigcap_{n>0}\bar{\mathbf{m}}^n$ is isomorphic to the field
$R/\bar{\mathbf{m}}\cong K$, so that $R$ is Artinian modulo the
infinite radical. In order to see that $R$ is not Artinian it will
be enough to check that the ascending chain

\begin{center}
$Rx_1\subseteq Rx_2\subseteq ...\subseteq Rx_n\subseteq ...$
\end{center}
is not stationary. Indeed if $Rx_n=Rx_{n+1}$ then there exists $a\in
R$ such that $x_{n+1}=ax_n$, and hence $x_{n+1}=ax_{n+1}x_{n+2}$.
This gives $x_{n+1}(1-ax_{n+2})=0$. But $1-ax_{n+2}$ is invertible
in $R$ since every power series of the form $1-X_{n+2}f$ is
invertible in $K[[X_1,X_2,...]]$. It follows that $x_{n+1}=0$. So
the equality $Rx_n=Rx_{n+1}=...$ implies that $x_{n+k}=0$, for all
$k>0$. By the first set of relations, this in turn implies that
$x_i=0$ for all $i>0$.

 The
proof will be finished if we prove that  $x_1\neq 0$. But if $x_1=0$
then in the successive sustitutions using the first set of
relations, we shall attain a power $x_n^{t}$, with $t>a_n$, for some
$n>0$. The successive sustitutions give
$x_1=x_2x_3=x_3^2x_4=x_4^3x_5^2=x_5^{5}x_6^3=...$. The $n$-th
expression  is of the form $x_n^{a_n}x_{n+1}^{a_{n-1}}$, for all
$n>0$,  convening that $a_0=0$. Therefore none of them is zero. That
ends the proof.
\end{proof}

\section{Rings with an epimorphic Noetherian envelope}

We start the section with  the following lemma.

\begin{lemma}
Let $A$ be a local Noetherian ring with maximal ideal $\mathbf{m}$.
Let $M$ be a (not necessarily finitely generated) $A$-module such
that $Supp\left(
M\right) =\left\{ \mathbf{m}\right\} $. If there exists a finite subset $%
\left\{ x_{1},x_{2},\ldots ,x_{r}\right\} \subset M$ such that $%
ann_{A}\left( M\right) =ann_{A}\left( x_{1},\ldots ,x_{r}\right) $, then $%
\mathbf{m}M\neq M$.
\end{lemma}

\begin{proof}
If $Supp\left( M\right) =\left\{ \mathbf{m}\right\} $ then
$Supp\left( Ax\right) =\left\{ \mathbf{m}\right\} $, for every $x\in
M\backslash \left\{ 0\right\} $. Fixing such an $x$, we have that
$\sqrt{ann_{A}\left( x\right) }=\mathbf{m}$ (cf. \cite[Theorem 6.6,
pg. 40]{Mat-89}). Then there exist an integer $n>0$ such that
$\mathbf{m}^{n}x=0$. It follows \ that
$M=\bigcup\limits_{n>0}ann_{M}\left( \mathbf{m}^{n}\right) $.

If now $\left\{ x_{1},x_{2},\ldots ,x_{r}\right\} \subset M$ is a finite
subset such that $ann_{A}\left( M\right) =ann_{A}\left( x_{1},\ldots
,x_{r}\right) $, then there exists a large enough $n>0$ such that $x_{i}\in
ann_{M}\left( \mathbf{m}^{n}\right) $ for $i=1,\ldots ,r$. Then $\mathbf{%
m}^{n}\subseteq ann_{A}\left( x_{1},\ldots ,x_{r}\right)
=ann_{A}\left( M\right) $, so that $\mathbf{m}^{n}M=0$. If we had
$\mathbf{m}M=M$ it would follow that $M=0$, which is impossible
since $Supp\left( M\right) \neq \varnothing $.
\end{proof}

We are now ready to prove the main result of this section. Given any module $%
M$, we denote by $M_{\mathbf{p}}$ the localization at the prime ideal $%
\mathbf{p}$.

\begin{theorem} \label{epimorphic Noetherian envelopes}
Let $R$ be a ring and $I$ an ideal of $R$ such that $R/I$ is
Noetherian. The following assertions are equivalent:

\begin{enumerate}
\item The projection $p:R\longrightarrow R/I$ is a Noetherian envelope;

\item $I$ is a nil ideal and $\mathbf{p}I_{\mathbf{p}}=I_{\mathbf{p}}$, for
all $\mathbf{p\in }Spec\left( R\right) $.
\end{enumerate}
\end{theorem}

\begin{proof}
$1)\Longrightarrow 2)$ By Proposition \ref{properties of  a
Noetherian preenvelope},  we have that $I=Ker\left( p\right)
\subseteq Nil\left( R\right) $ and so $I$ is a nil ideal.

Let know $\mathbf{p}$ be a prime ideal of $R$. Then, by the proof of assertion 2 in  Proposition
\ref{properties of  a Noetherian preenvelope},  we see that the canonical projection $\pi =p_{\mathbf{p}}:R_{%
\mathbf{p}}\longrightarrow R_{\mathbf{p}}/I_{\mathbf{p}}$ is also a
Noetherian envelope. So, replacing $R$ and $I$ by $R_{\mathbf{p}}$ and $%
I_{\mathbf{p}}$ respectively, we can assume that \ $R$ is local
(with maximal ideal $\mathbf{m}$) and have to prove that
$\mathbf{m}I=I$. Indeed, if $\mathbf{m}I\neq I$ then $I/\mathbf{m}I$
(and hence $I$) has simple quotients, which contradicts Lemma
\ref{no intermediate in epimorphic envelope}.

$2)\Longrightarrow 1)$ We have to prove that if $f:R\longrightarrow
N$ is a ring homomorphism, with $N$ Noetherian, then $f\left(
I\right) =0$. Suppose that is not true and fix an $f$ such that
$f\left( I\right) \neq 0$. Note that $f\left( Nil\left( R\right)
\right) N$ is a nil ideal of the Noetherian ring $N$. It follows
that $f\left( Nil\left( R\right) \right) N$ is nilpotent and, as a
consequence, that $f\left( Nil\left( R\right) I\right) N=f\left(
Nil\left( R\right) \right) Nf\left( I\right) N\neq f\left( I\right)
N$ for otherwise we would get $f\left( I\right) N=0$ and hence
$f\left( I\right) =0$, against the assumption.

We next consider the composition%
\begin{equation*}
I\overset{f}{\longrightarrow }f\left( I\right) N\longrightarrow f\left(
I\right) N/f\left( Nil\left( R\right) I\right) N.
\end{equation*}%
Its kernel is $I\cap f^{-1}\left( f\left( Nil\left( R\right) I\right)
N\right) $ and we get a monomorphism of $R$-modules%
\begin{equation*}
\overline{f}:M=:I/\left[ I\cap f^{-1}\left( f\left( Nil\left( R\right)
I\right) N\right) \right] \longrightarrow f\left( I\right) N/f\left(
Nil\left( R\right) I\right) N.
\end{equation*}%
Note that $M\neq 0$ for otherwise we would have $f\left( I\right) \subseteq
f\left( Nil\left( R\right) I\right) N$ and hence $f\left( I\right) N=f\left(
Nil\left( R\right) I\right) N$, that we have seen that is impossible.

On the other hand, \ since $I$ is a nil ideal, we have that $I\subseteq
Nil\left( R\right) $ and hence $R_{red}\cong \left( R/I\right) _{red}$ is a
Noetherian ring. We shall view this latter isomorphism as an identification
and put $A=R_{red}$ in the sequel. Since we clearly have $Nil\left( R\right)
M=0$ it follows that $M$ is an $A$-module in the canonical way. Moreover $%
Im\left( \overline{f}\right) $ generates the finitely generated $N$-module $%
f\left( I\right) N/f\left( Nil\left( R\right) I\right) N$, which allows us
to choose a finite subset $\left\{ x_{1},x_{2},\ldots ,x_{r}\right\} \subset
M$ such that $\left\{ \overline{f}\left( x_{1}\right) ,\ldots ,\overline{f}%
\left( x_{r}\right) \right\} $ generates $f\left( I\right) N/f\left(
Nil\left( R\right) I\right) N$ (as an $N$-module). We shall derive from that
that $ann_{A}\left( M\right) =ann_{A}\left( x_{1},\ldots ,x_{r}\right) $.
Indeed if $\overline{r}=r+Nil\left( R\right) $ is an element of $A$ such
that $rx_{i}=0$ for $i=1,\ldots ,r,$ then $f\left( r\right) \overline{f}%
\left( x_{i}\right) =0$ for all $i=1,\ldots ,r.$ It follows that $f\left(
r\right) f\left( I\right) N\subseteq f\left( Nil\left( R\right) I\right) N$,
so that $f\left( rI\right) \subseteq f\left( Nil\left( R\right) I\right) N$
and hence $rI\subseteq f^{-1}\left( f\left( Nil\left( R\right) I\right)
N\right) $. This implies that $\overline{r}M=rM=0$.

We claim now that%
\begin{equation*}
Supp_{A}\left( M\right) =\mathcal{V}\left( ann_{A}\left( M\right) \right)
=\left\{ \mathbf{q}\in Spec\left( A\right) :ann_{A}\left( M\right) \subseteq
\mathbf{q}\right\} .
\end{equation*}%
Indeed if $\mathbf{q\in }$ $Spec\left( A\right) $ and $ann_{A}\left(
x_{1},\ldots ,x_{r}\right) =ann_{A}\left( M\right) \nsubseteqq \mathbf{q}$,
then we can find an element $s\in A\backslash \mathbf{q}$ such that $%
sx_{i}=0 $, for $i=1,\ldots ,r$, and hence such that $sM=0$. It follows that
$M_{\mathbf{q}}=0$ and so $\mathbf{q}\notin Supp\left( M\right) $. That
proves that $Supp\left( M\right) \subseteq $ $\mathcal{V}\left(
ann_{A}\left( M\right) \right) $. On the other hand, if $\mathbf{q}\notin
Supp\left( M\right) $ then, for every $i=1,\ldots ,r,$ we can find an
element $s_{i}\in A\backslash \mathbf{q}$ such that $s_{i}x_{i}=0$. Then $%
s=s_{1}\cdot \cdots \cdot s_{r}$ belongs to $ann_{A}\left( x_{1},\ldots
,x_{r}\right) $ $=ann_{A}\left( M\right) $, so that $ann_{A}\left( M\right)
\subsetneq \mathbf{q}$ and hence $\mathbf{q}\notin \mathcal{V}\left(
ann_{A}\left( M\right) \right) $.

The equality $Supp_{A}\left( M\right) =\mathcal{V}\left( ann_{A}\left(
M\right) \right) $ and the fact that we are assuming $M\neq 0$ (and hence $%
ann_{A}\left( M\right) \neq A$) imply that we can pick up a prime ideal $%
\mathbf{q}$ of $A$ which is minimal among those containing $ann_{A}\left(
M\right) $, and thereby minimal in $Supp_{A}\left( M\right) $. We localize
at $\mathbf{q}$ and obtain a module $M_{\mathbf{q}}$ over the local ring $A_{%
\mathbf{q}}$ such that $Supp_{A_{\mathbf{q}}}\left( M_{\mathbf{q}%
}\right) =\left\{ \mathbf{q}A_{\mathbf{q}}\right\} $. Moreover the finite
subset $\left\{ x_{1},\ldots ,x_{r}\right\} \subseteq M_{\mathbf{q}}$
satisfies that $ann_{A_{\mathbf{q}}}\left( x_{1},\ldots ,x_{r}\right)
=ann_{A_{\mathbf{q}}}\left( M_{\mathbf{q}}\right) $. Indeed if $a/s\in A_{%
\mathbf{q}}$ satisfies that $ax_{i}/s=0$ in $M_{\mathbf{q}}$, for all $%
i=1,\ldots ,r$, then we can find an element $t\in A\backslash \mathbf{q}$
such that $tax_{i}=0,$ for $i=1,\ldots ,r$. Since $ann_{A}\left( M\right)
=ann_{A}\left( x_{1},\ldots ,x_{r}\right) $ it follows that $ta\in
ann_{A}\left( M\right) $ and, hence, that $a/s=ta/ts\in ann_{\acute{A}_{%
\mathbf{q}}}\left( M_{\mathbf{q}}\right) $. Now we can apply Lemma 0.4 to
the local Noetherian ring $A_{\mathbf{q}}$ and the $A_{\mathbf{q}}$-module $%
M_{\mathbf{q}}$.  We conclude that $\mathbf{q}M_{\mathbf{q}}\neq M_{\mathbf{q}%
}.$

We take now the prime ideal $\mathbf{p}$ of $R$ such that $\mathbf{p/}%
Nil\left( R\right) =\mathbf{q}$. From the equality $\mathbf{p}I_{\mathbf{p}%
}=I_{\mathbf{p}}$ it follows that $\mathbf{p}M_{\mathbf{p}}=M_{\mathbf{p}}$.
We will derive that $M_{\mathbf{q}}=\mathbf{q}M_{\mathbf{q}}$, thus getting
a contradiction and ending the proof. Let $x\in M$ be any element. Since $%
x\in \mathbf{p}M_{\mathbf{p}}$, we have an equality $x=\sum\limits_{1\leq
j\leq m}p_{j}y_{j}/s_{j}$, where $p_{j}\in \mathbf{p}$, $y_{j}\in M$ and $%
s_{j}\in R\backslash \mathbf{p}$. Multiplying by $s=s_{1}\cdots s_{m}$, we
see that $sx\in \mathbf{p}M,$ which is equivalent to sayt that $\overline{s}%
x\in \mathbf{q}M$, where $\overline{s}=s+Nil\left( R\right) \in A\backslash
\left\{ \mathbf{q}\right\} $. Then we have that $x=\overline{s}x/\overline{s}%
\in \mathbf{q}M_{\mathbf{q}}$, for every $x\in M$, which implies that $M_{%
\mathbf{q}}=\mathbf{q}M_{\mathbf{q}}$ as desired.
\end{proof}

Recall that the \textbf{trivial extension }of a ring $A$ by the $A$-module $%
N $, denoted by $R=A\rtimes N$, has as underlying additive abelian group $%
A\oplus N$ and the multiplication is defined by the rule $\left( a,m\right)
\cdot \left( b,n\right) =\left( ab,an+bm\right) $.

\begin{example} \label{trivial extensions}
\begin{enumerate}
\item Let $A$ be a Noetherian integral domain, $X$ be a finitely generated $%
A $-module and $D$ a torsion divisible $A$-module. Put $N=X\oplus D$
and take $R=A\rtimes N$. The ideal $I=0\rtimes D=\{(a,n) \in R:$
$a=0\text{ and }n\in D\}$  satisfies condition 2 in the above
theorem and that $R/I$ is Noetherian.
Therefore $%
p:R\longrightarrow R/I\cong A\rtimes X$ is the Noetherian envelope.

\item If in the above example we do not assume  $D$ to be torsion then,
with the same choice \ of $I$, the projection $p:R\longrightarrow R/I$ is
not a Noetherian envelope.
\end{enumerate}
\end{example}

\begin{proof}
\begin{enumerate}
\item $Nil(R)=0\rtimes N$ is contained in any prime ideal of $R$,
which implies that any such prime ideal is of the form
$\widehat{\mathbf{p}}=\mathbf{p}\rtimes N$, where $\mathbf{p}\in
Spec\left( A\right) $. Now one check the following equalities:

\begin{enumerate}
\item $R_{\widehat{\mathbf{p}}}=A_{\mathbf{p}}\rtimes N_{\mathbf{p}}=A_{%
\mathbf{p}}\rtimes \left( X_{\mathbf{p}}\oplus D_{\mathbf{p}}\right) $

\item $I_{\widehat{\mathbf{p}}}=0\rtimes \left( 0\oplus D_{\mathbf{p}%
}\right) $

\item $\widehat{\mathbf{p}}I_{\widehat{\mathbf{p}}}=0\rtimes \left( 0\oplus
\mathbf{p}D_{\mathbf{p}}\right) $.

The divisibility of $D$ gives that $D=\mathbf{p}D,$ for every $\mathbf{p}\in
Spec\left( A\right) \backslash \left\{ 0\right\} $, while we have that $%
D_{0} $ (=localization at $\mathbf{p=0}$) is zero due to the fact that $D$
is a torsion $A$-module. That proves that $\widehat{\mathbf{p}}I_{\widehat{%
\mathbf{p}}}=I_{\widehat{\mathbf{p}}}$, for all $\widehat{\mathbf{p}}\in
Spec\left( R\right) $.
\end{enumerate}
\end{enumerate}

2) The argument of the above paragraph shows that if $D$ is not torsion (and
hence $D_{0}\neq 0$) then $I$ does not satifies condition 2 of the Theorem.
\end{proof}

\section{Does there exist a non-Noetherian ring with a monomorphic
Noetherian envelope?}

If $f:R\longrightarrow N$ is a Noetherian (pre)envelope of the ring
$R$, then the inclusion $R':=Im(f)\hookrightarrow N$ is a
monomorphic Noetherian (pre)envelope. In order to identify the rings
having a Noetherian envelope, one needs to identify those having a
monomorphic Noetherian envelope. That makes pertinent the question
in the title of this section, which we address from now on. We start
with the following result:

\begin{proposition} \label{minimal counterexample}
Suppose that there exists a non-Noetherian ring  having a
monomorphic Noetherian preenvelope. Then there is a non-Noetherian
local ring $R$ (with maximal ideal $\mathbf{m}$) having a
monomorphic Noetherian preenvelope $j:R\hookrightarrow N$ satisfying
the following properties:

\begin{enumerate}
\item $R$ has finite Krull dimension $K-dim(R)=d>0$, and every ring of Krull dimension $<d$ having a monomorphic Noetherian preenvelope is
Noetherian \item $R_{red}$ is a Noetherian ring
\item If $0\neq I\subseteq Nil(R)$ is an ideal such that
$R/I$ is non-Noetherian, then $D:=\frac{IN\cap R}{I}$ is an
$R$-module such that $Supp(D)=\{\mathbf{m}\}$ and $\mathbf{m}D=D$.
\end{enumerate}
\end{proposition}
\begin{proof}
If $A$ is a non-Noetherian ring with a monomorphic Noetherian
preenvelope $i:A\rightarrowtail N'$, then the set of restricted
ideals $I'\subseteq A$ such that  $A/I'$ is non-Noetherian has a
maximal element, say, $J$. Then $R=A/J$ is a non-Noetherian ring
having a monomorphic Noetherian preenvelope
$j=\bar{i}:R=A/J\rightarrowtail N'/N'J=:N$ with the property that, a
for a  nonzero ideal $I$ of $R$, the following three assertions are
equivalent:

\begin{enumerate}
\item[i)] $R/I$ is Noetherian \item[ii)] $I=IN\cap R$ \item[iii)] $R/I$
has a monomorphic Noetherian preenvelope
\end{enumerate}
We claim that there is a maximal ideal $\mathbf{m}$ of $R$ such that
$R_\mathbf{m}$ is not Noetherian. Bearing in mind that
$j_\mathbf{m}:R_\mathbf{m}\longrightarrow N_\mathbf{m}$ is also a
(monomorphic) Noetherian preenvelope (see the proof of Proposition
\ref{properties of  a Noetherian preenvelope}(2)),  it will follow
that, after replacing $R$ by an appropriate factor of
$R_\mathbf{m}$, one gets a non-Noetherian local ring with a
monomorphic Noetherian preenvelope satisfying that the properties
i)-iii) above are also equivalent for it.

Suppose our claim is false, so that $R_\mathbf{m}$ is Noetherian for
all $\mathbf{m}\in Max(R)$. Let then $I_0\subsetneq I_1\subsetneq
...$ be a strictly increasing chain of ideals in $R$. We can assume
that $0\neq I_0=:I$. Then $R/I$ cannot be a Noetherian ring. Since
the induced map $\bar{j}:R/I\longrightarrow N/NI$ is a Noetherian
preenvelope it follows that $\bar{j}$ is not injective, and hence
$0\neq\hat{I}/I$, where $\hat{I}:=R\cap NI$. But $N\hat{I}=NI$ and
the induced map $R/\hat{I}\longrightarrow N/N\hat{I}$ is a
monomorphic Noetherian preenvelope. Our assumptions on $R$ imply
that $R/\hat{I}$ is Noetherian, so that the canonical projection
$p:R/I\twoheadrightarrow R/\hat{I}$ is a Noetherian envelope.
According to Theorem \ref{epimorphic Noetherian envelopes}, we have
that
$\mathbf{m}(\frac{\hat{I}}{I})_\mathbf{m}=(\frac{\hat{I}}{I})_\mathbf{m}$
(*) for every $\mathbf{m}\in Max(R)$. The fact that $R_\mathbf{m}$
is Noetherian implies that
$(\frac{\hat{I}}{I})_\mathbf{m}=\frac{\hat{I}_\mathbf{m}}{I_\mathbf{m}}$
is a finitely generated $R_\mathbf{m}$-module. Then, using
Nakayama's lemma, from the equality (*) we get that
$(\frac{\hat{I}}{I})_\mathbf{m}=0$, for all $\mathbf{m}\in Max(R)$.
It follows that $\hat{I}/I=0$, and we then get a contradiction.

So, from now on in this proof, we assume that $R$ is a local
non-Noetherian ring having a monomorphic Noetherian preenvelope, for
which condition i)-iii) are equivalent. By Proposition
\ref{properties of  a Noetherian preenvelope}, every $\mathbf{p}\in
Spec(R)$ is restricted and therefore $R/\mathbf{p}$ is Noetherian
for all $\mathbf{p}\in Spec(R)$. Since  there are only finitely many
minimal elements in $Spec(R)$ (cf. Proposition \ref{properties of  a
Noetherian preenvelope}) we conclude that $R_{red}$ is a Noetherian
ring and, hence, that $K-dim(R)<\propto$. It implies, in particular,
that one could have chosen our initial ring $A$ with minimal finite
Krull dimension. Having done so,  this final local ring $R$ is has
also minimal finite Krull dimension among the non-Noetherian rings
having a monomorphic Noetherian preenvelope. In particular, we have
that $R_\mathbf{p}$ is Noetherian, for every $\mathbf{p}\in
Spec(R)\setminus\{\mathbf{m}\}$.

Finally, if $0\neq I\subset Nil(R)$ is an ideal such that $R/I$ is
not Noetherian (i.e. $I\subsetneq\hat{I}:=R\cap NI$), the third
paragraph of this proof shows that the canonical projection
$p:R/I\twoheadrightarrow R/\hat{I}$ is a Noetherian envelope. Then
Theorem \ref{epimorphic Noetherian envelopes} says that
$D=\hat{I}/I$ has the property that
$\mathbf{p}D_\mathbf{p}=D_\mathbf{p}$, for all $\mathbf{p}\in
Spec(R)$. But then $D_\mathbf{p}=(\frac{\hat{I}}{I})_\mathbf{p}=0$,
for every non-maximal $\mathbf{p}\in Spec(R)$, because
$R_\mathbf{p}$ is a Noetherian ring. It follows that
$Supp(D)=\{\mathbf{m}\}$ and that $\mathbf{m}D=D$.
\end{proof}

\begin{example} \label{possible counterexample?}
Let $\mathbf{Z}_{(p)}$ denote the localization of $\mathbf{Z}$ at
the prime ideal $(p)=p\mathbf{Z}$ and consider the trivial extension
$R=\mathbf{Z}_{(p)}\rtimes\mathbf{Q}$. Then $R$ is a non-Noetherian
local ring and, in case of having a Noetherian preenvelope, this
would be monomorphic and conditions 1)-3) of the above proposition
would hold.
\end{example}
\begin{proof}
Since $0\times M$ is an ideal of $R$ for each
$\mathbf{Z}_{(p)}$-submodule  $M$ of $\mathbf{Q}$ it follows that
$R$ is not Noetherian. The prime ideals of $R$ are
$p\mathbf{Z}_{(p)}\rtimes\mathbf{Q}$ and
$0\rtimes\mathbf{Q}=Nil(R)$, so that $R$ is local with maximal ideal
$\mathbf{m}:=p\mathbf{Z}_{(p)}\rtimes\mathbf{Q}$ and $K-dim(R)=1$.
In particular, condition 1) Proposition \ref{minimal counterexample}
is satisfied (see Theorem \ref{artinian}).  Since $R$ is a subring
of the Noetherian ring
$\mathbf{Q}\rtimes\mathbf{Q}\cong\mathbf{Q}[x]/(x^2)$, any
Noetherian preenvelope $j:R\longrightarrow N$ that might exist would
be necessarily monomorphic. On the other hand
$R_{red}\cong\mathbf{Z}_{(p)}$ is a Noetherian ring.

Finally, suppose that $j:R\longrightarrow N$ is a Noetherian
preenvelope, which we view as an inclusion, and  let $0\neq
I\subseteq Nil(R)$ be an ideal of $R$ such that $R/I$ is
non-Noetherian. We have that $I=0\rtimes A$ and
$R/I\cong\mathbf{Z}_{(p)}\rtimes (\mathbf{Q}/A)$, for some
$\mathbf{Z}_{(p)}$-submodule $0\neq A\subsetneq\mathbf{Q}$. Note
that $\hat{I}=R\cap NI$ consists of nilpotent elements, so that
$\hat{I}=0\rtimes B$, for some $\mathbf{Z}_{(p)}$-submodule
$A\subseteq B\subseteq\mathbf{Q}$. We need to prove that
$B=\mathbf{Q}$, and then condition 3) of Proposition \ref{minimal
counterexample} will be automatically satisfied.

Indeed, on one side we have that the induced map
$\tilde{j}:R/\hat{I}\longrightarrow N/NI$ is a monomorphic
Noetherian preenvelope. But in case $B\subsetneq\mathbf{Q}$, we have
$R/\hat{I}\cong\mathbf{Z}_{(p)}\rtimes (\mathbf{Q}/B)$ and Example
\ref{trivial extensions} says that the canonical projection  $\pi
:R/\hat{I}\cong \mathbf{Z}_{(p)}\rtimes
(\mathbf{Q}/B)\twoheadrightarrow\mathbf{Z}_{(p)}$ is the Noetherian
envelope. This is absurd for then we would have  $0\neq 0\rtimes
(\mathbf{Q}/B)=Ker(\pi )\subseteq Ker(\tilde{j})=0$.
\end{proof}

The last proposition and example propose the ring
$R=\mathbf{Z}_{(p)}\rtimes\mathbf{Q}$ as an obvious candidate to be
a 'minimal' non-Noetherian ring having a monomorphic Noetherian
preenvelope. We have the following result.

\begin{theorem} \label{nonexistence of Noetherian envelope}
The ring $\mathbf{Z}_{(p)}\rtimes\mathbf{Q}$ does not have a
Noetherian envelope.
\end{theorem}

The proof of this theorem will cover the rest of the paper and is
based on a few lemmas.We proceed by reduction to absurd and,
\underline{in the sequel},  \underline{we assume that}
$i:\mathbf{Z}_{(p)}\rtimes\mathbf{Q}\hookrightarrow N$ \underline{is
a monomorphic Noetherian envelope}, which we view  as an inclusion.
Recall that a ring $A$ is called {\bf indecomposable} when it cannot
be properly decomposed as a product $A_1\times A_2$ of two rings.
That is equivalent to say that the only idempotent elements of $A$
are $0$ and $1$.

\begin{lemma} \label{decomposition of the envelope}
There is a ring isomorphism

\begin{center}
$\varphi :N\stackrel{\cong}{\longrightarrow}\mathbf{Z}_{(p)}\times
B_2\times ...\times B_r$
\end{center}
satisfying the following properties:

\begin{enumerate}
\item $\varphi i$ is a matrix map
$\begin{pmatrix} \pi\\
\lambda_2\\.\\.\\\lambda_r\end{pmatrix}:\mathbf{Z}_{(p)}\rtimes\mathbf{Q}
\longrightarrow\mathbf{Z}_{(p)}\times B_2\times ...\times B_r$,
where $\pi
:\mathbf{Z}_{(p)}\rtimes\mathbf{Q}\twoheadrightarrow\mathbf{Z}_{(p)}$
is the projection and each $\lambda_i$  is an injective ring
homomorphism into the indecomposable Noetherian ring $B_i$
\item Each $\lambda_i$ is a minimal morphism in $CRings$  \item If $\mu
:\mathbf{Z}_{(p)}\rtimes\mathbf{Q}\rightarrowtail S$ is an injective
ring homomorphism, with $S$ an indecomposable Noetherian ring, then
$\mu$ factors through some $\lambda_i$ \item There is no ring
homomorphism $h:B_i\longrightarrow B_j$, with $i\neq j$, such that
$h\lambda_i=\lambda_j$ \item The ring $B_i$ does not contain a
proper Noetherian subring containing $Im(\lambda_i)$.
\end{enumerate}
\end{lemma}
\begin{proof}
 A simple
observation will be frequently used, namely, that there cannot exist
a proper Noetherian subring $B$ of $N$ containing $R$ as a subring.
Indeed, if such $B$ exists and $u:R\hookrightarrow B$ is the
inclusion, then we get a ring homomorphism $g:N\longrightarrow B$
such that $gi=u$. Then the composition
$h:N\stackrel{g}{\longrightarrow}B\hookrightarrow N$ is a
non-bijective ring homomorphism such that $hi=i$, against the fact
that $i$ is an envelope.

The projection $\pi :R\twoheadrightarrow\mathbf{Z}_{(p)}$ is a
retraction in the category  $CRings$ of commutative rings. Moreover,
since $i$ is a Noetherian envelope, we have a ring homomorphism
$f:N\longrightarrow\mathbf{Z}_{(p)}$ such that $fi=\pi$. It follows
that also $f$ is a retraction in $CRings$, so that we have a
$\mathbf{Z}_{(p)}$-module decomposition $N=\mathbf{Z}_{(p)}\oplus
I$, where $I:=Ker(f)$ is an ideal of $N$ containing
$i(0\rtimes\mathbf{Q})$.

Also due to the fact that $i$ is a Noetherian envelope, we have a
factorization of the inclusion
$j:R\stackrel{i}{\longrightarrow}N\stackrel{\rho}{\longrightarrow}\mathbf{Q}\rtimes\mathbf{Q}$
in $CRings$. Then $Im(\rho )$ is a subring of
$\mathbf{Q}\rtimes\mathbf{Q}$ containing
$R=\mathbf{Z}_{(p)}\rtimes\mathbf{Q}$. One readily sees that
$Im(\rho )=A\rtimes\mathbf{Q}$, where $A$ is a subring of
$\mathbf{Q}$ containing $\mathbf{Z}_{(p)}$ as a subring. But if
$q\in\mathbf{Q}\setminus\mathbf{Z}_{(p)}$, then we can write
$q=ap^{-t}$, for some invertible element $a\in\mathbf{Z}_{(p)})$ and
some integer $t>0$. Then we have
$\mathbf{Z}_{(p)}[q]=\mathbf{Z}_{(p)}[p^{-t}]$. Given any integer
$n>0$, Euclidean division gives that $n=tm+r$, where $m>0$ and
$0<r<t$. We then have $p^{-n}=p^{-r}\cdot
(p^{-t})^m=p^{t-r}(p^{-t})^{m+1}$, which proves that
$p^{-n}\in\mathbf{Z}_{(p)}[p^{-t}]=\mathbf{Z}_{(p)}[q]$, for every
$n>0$, and hence that $\mathbf{Z}_{(p)}[q]=\mathbf{Q}$. As a
consequence, we get that either $A=\mathbf{Z}_{(p)}$ or
$A=\mathbf{Q}$, and so that either $Im(\rho
)=\mathbf{Z}_{(p)}\rtimes\mathbf{Q}=R$ or $Im(\rho
)=\mathbf{Q}\rtimes\mathbf{Q}$. But the first possibility is
discarded for, being a factor of Noetherian, the ring $Im(\rho )$ is
Noetherian. Therefore any ring homomorphism $\rho
:N\longrightarrow\mathbf{Q}\rtimes\mathbf{Q}$ such that $\rho i=j$
is necessarily surjective.

We fix such a $\rho$ from now on and also fix the decomposition
$N=\mathbf{Z}_{(p)}\oplus I$ considered above. We claim that the
restriction of $\rho$

\begin{center}
$\rho_{|I}:I\longrightarrow\mathbf{Q}\rtimes\mathbf{Q}$
\end{center}
is a surjective map. Indeed $\rho (I)$ is a nonzero ideal of
$\mathbf{Q}\rtimes\mathbf{Q}$ since $\rho$ is a surjective ring
homomorphism. Then we get that either $\rho (I)=0\rtimes\mathbf{Q}$
or $\rho (I)=\mathbf{Q}\rtimes\mathbf{Q}$. But the first possibility
is discarded for it would produce a surjective ring homomorphism

\begin{center}
$\bar{\rho}:\mathbf{Z}_{(p)}\cong
N/I\twoheadrightarrow\frac{\mathbf{Q}\rtimes\mathbf{Q}}{0\rtimes\mathbf{Q}}\cong\mathbf{Q}$.
\end{center}

Next we claim that $0\rtimes\mathbf{Q}\subset
(0\rtimes\mathbf{Q})I$, which will imply that
$0\rtimes\mathbf{Q}\subset I^2$ and, hence, that
$\mathbf{Z}_{(p)}+I^2$ is a subring of $N$ containing $R$ as a
subring. Indeed we have
$(0\rtimes\mathbf{Q})N=(0\rtimes\mathbf{Q})(\mathbf{Z}_{(p)}\oplus
I)=(0\rtimes\mathbf{Q})+(0\rtimes\mathbf{Q})I$ and, if our claim
were not true, we would get:

\begin{center}
$0\neq\frac{(0\rtimes\mathbf{Q})N}{(0\rtimes\mathbf{Q})I}\cong
\frac{(0\rtimes\mathbf{Q})+(0\rtimes\mathbf{Q})I}{(0\rtimes\mathbf{Q})I}\cong\frac{\mathbf{Q}}{X}$,
\end{center}
where $X$ is the $\mathbf{Z}_{(p)}$-submodule of $\mathbf{Q}$
consisting of those of those $q\in\mathbf{Q}$ such that $(0,q)\in
(0\rtimes\mathbf{Q})I$. It is routinary to see that the isomorphism
of $\mathbf{Z}_{(p)}$-modules
$\frac{(0\rtimes\mathbf{Q})N}{(0\rtimes\mathbf{Q})I}\cong\mathbf{Q}/X$
gives a bijection between the $N$-submodules of
$\frac{(0\rtimes\mathbf{Q})N}{(0\rtimes\mathbf{Q})I}$ and the
$\mathbf{Z}_{(p)}$-submodules of $\mathbf{Q}/X$. But this is
impossible for, due to the Noetherian condition of $N$, the
$N$-module $\frac{(0\rtimes\mathbf{Q})N}{(0\rtimes\mathbf{Q})I}$ is
Noetherian while $\mathbf{Q}/X$ does not satisfies ACC on
$\mathbf{Z}_{(p)}$-submodules.

We next consider the subring $N'=\mathbf{Z}_{(p)}+I^2$ of $N$. If
$\{y_1,...,y_r\}$ is a finite set of generators of $I$ as an ideal,
then $\{1,y_1,...,y_r\}$ generates $N=\mathbf{Z}_{(p)}+I$ as a
$N'$-module. By Eakin's theorem (cf. \cite{Eakin-68}, see also
\cite{FJ-74}), we know that $N'$ is a Noetherian ring. By the first
paragraph of this proof, we conclude that $N'=N$, from which it
easily follows that $I^2=I$. But then there is an idempotent element
$e=e^2\in N$ such that $I=Ne$ (cf. \cite{AF-92}[Exercise 7.12]).
Since $y(1-e)=0$ for all $y\in I$, we get that $(a+y)(1-e)=a(1-e)$,
for all $a\in\mathbf{Z}_{(p)}$ and $y\in I$. Therefore
$A:=N(1-e)=\mathbf{Z}_{(p)}(1-e)$ is a ring (with unit $1-e$)
isomorphic to $\mathbf{Z}_{(p)}$ via the assignment
$a\rightsquigarrow a(1-e)$. We put $B:=I=Ne$, which is a ring (with
unit $e$), and we have a ring isomorphism $\varphi
:N\stackrel{\cong}{\longrightarrow} \mathbf{Z}_{(p)}\times B$.
Bearing in mind that $N=\mathbf{Z}_{(p)}\oplus I$ as
$\mathbf{Z}_{(p)}$-modules, it is easy to see that $\varphi
(a+b)=(a,ae+b)$, for all $a\in\mathbf{Z}_{(p)}$ and $b\in B=I$.

Then the composition
$i':\mathbf{Z}_{(p)}\rtimes\mathbf{Q}\stackrel{i}{\hookrightarrow}N\stackrel{\varphi}{\longrightarrow}
\mathbf{Z}_{(p)}\rtimes B$ is also a Noetherian envelope. Its two
component maps are:

\begin{center}
$\pi
:\mathbf{Z}_{(p)}\rtimes\mathbf{Q}\twoheadrightarrow\mathbf{Z}_{(p)}$,
\hspace*{0.5cm} $(a,q)\rightsquigarrow a$

$\lambda :\mathbf{Z}_{(p)}\rtimes\mathbf{Q}\longrightarrow B$,
\hspace*{0.5cm} $(a,q)\rightsquigarrow ae+(0,q)$.
\end{center}

We first note that every non-injective ring homomorphism
$g:R\longrightarrow S$, where $S$ is Noetherian indecomposable,
factors through $\pi$. Indeed if $Ker(g)$ contains
$0\rtimes\mathbf{Q}$ that is clear. In any other case, we have
$Ker(g)=0\rtimes M$, for some nonzero $\mathbf{Z}_{(p)}$-submodule
$M$ of $\mathbf{Q}$. Then the induced monomorphism
$\frac{\mathbf{Z}_{(p)}\rtimes\mathbf{Q}}{0\rtimes M}\cong
\mathbf{Z}_{(p)}\rtimes (\mathbf{Q}/M)\longrightarrow S$ factors
through the Noetherian envelope of $\mathbf{Z}_{(p)}\rtimes
(\mathbf{Q}/M)$ which, by Example \ref{trivial extensions}, is the
projection $\mathbf{Z}_{(p)}\rtimes
(\mathbf{Q}/M)\twoheadrightarrow\mathbf{Z}_{(p)}$. It then follows
that $g$ factors through $\pi$ as desired.

Decompose now $B$ as a finite product of indecomposable (Noetherian)
rings $B=B_2\times ...\times B_r$. Then $\lambda$ is identified with
a matrix map

\begin{center}
$\begin{pmatrix} \lambda_2\\ .\\ .\\
\lambda_r\end{pmatrix}:R\longrightarrow B_2\times ...\times B_r$,
\end{center}
where the $\lambda_i$ are ring homomorphisms. We claim that these
$\lambda_i$ are necessarily injective, thus proving property 1) in
the statement. Indeed if, say, $\lambda_2$ is not injective then, by
the previous paragraph, we have $\lambda_2=u\pi$ for some ring
homomorphism $u:\mathbf{Z}_{(p)}\longrightarrow B_2$. Now from the
ring endomorphism $\xi =\begin{pmatrix} 1 & 0\\ u & 0\end{pmatrix}
:\mathbf{Z}_{(p)}\times B_2\longrightarrow\mathbf{Z}_{(p)}\times
B_2$ we derive a ring endomorphism

\begin{center}
$\Phi =\begin{pmatrix}\xi & 0\\ 0 & 1
\end{pmatrix}:(\mathbf{Z}_{(p)}\times B_2)\times (B_3\times ...\times B_r)\longrightarrow
(\mathbf{Z}_{(p)}\times B_2)\times (B_3\times ...\times B_r)$
\end{center}
which is not bijective and satisfies that $\Phi i'=i'$ (here
$1:\prod_{3\leq i\leq r}B_i\longrightarrow\prod_{3\leq i\leq r}B_i$
is the identity map). That would contradict the fact that $i'$ is an
envelope.

 Properties 2) and 3) in the statement will follow easily once we check the following two properties for
 $\lambda$:

\begin{enumerate}
\item[a)] $\lambda$ is minimal, i.e. if $g:B\longrightarrow
B$ is a ring homomorphism such that $g\lambda =\lambda$, then
$\lambda$ is an isomorphism \item[b)] If $\mu
:\mathbf{Z}_{(p)}\rtimes\mathbf{Q}\longrightarrow S$ is any
injective ring homomorphism, with $S$ an indecomposable Noetherian
ring, then $\mu$ factors through $\lambda$.
\end{enumerate}
 Indeed, let  $g:B\longrightarrow B$ be a ring endomorphism such that $g\lambda
=\lambda$, then the 'diagonal' map $\psi :=\begin{pmatrix} 1 & 0\\ 0
& g
\end{pmatrix}:\mathbf{Z}_{(p)}\times B\longrightarrow \mathbf{Z}_{(p)}\times
B$ is a ring endomorphism such that $\psi i'=i'$. It follows that
$\psi$ is an isomorphism and, hence, that $g$ is an isomorphism.
Finally, if $\mu :\mathbf{Z}_{(p)}\rtimes\mathbf{Q}\rightarrowtail
S$ is an injective ring homomorphism, with $S$ an indecomposable
Noetherian ring, then the fact that $i'$ is a Noetherian envelope
gives a ring homomorphism $\upsilon =\begin{pmatrix}\upsilon_1 &
\upsilon_2
\end{pmatrix} :\mathbf{Z}_{(p)}\times B\longrightarrow S$ such that
$\upsilon i'=\mu$. The fact that $S$ is indecomposable implies that
either $\upsilon_1=0$ or $\upsilon_2=0$. But the second possibility
is discarded for it would imply that $\mu =\upsilon_1\pi$, and so
that $0\rtimes\mathbf{Q}=Ker(\pi )\subseteq Ker(\mu )=0$.

It only remains to prove properties 4) and 5) in the statement. To
prove 4), take any ring homomorphism $h:B_i\longrightarrow B_j$,
with $i\neq j$. Without loss of generality, put $i=2$ and $j=3$. If
$h\lambda_2=\lambda_3$, then we consider the ring homomorphism given
matricially in the form

\begin{center}
$\Psi =\begin{pmatrix} \psi & 0\\ 0 & 1
\end{pmatrix}$,
\end{center}
where $\psi =\begin{pmatrix} 1 & 0 & 0\\ 0 & 1 & 0\\ 0 & h &
0\end{pmatrix}:\mathbf{Z}_{(p)}\times B_2\times B_3\longrightarrow
\mathbf{Z}_{(p)}\times B_2\times B_3$ and $1:\oplus_{4\leq i\leq
r}B_i\longrightarrow \oplus_{4\leq i\leq r}B_i$ is the identity map.
We have an equality $\Psi i'=i'$, but $\Psi$ is not an isomorphism,
which contradicts the fact that $i'$ is an envelope.

Finally, suppose that $N'\subsetneq B_i$ is a proper Noetherian
subring containing $Im(\lambda_i)$. Then, putting $i=2$ for
simplicity, we get that $\mathbf{Z}_{(p)}\times N'\times B_3\times
...\times B_r$ is a proper Noetherian subring of
$\mathbf{Z}_{(p)}\times B_2\times B_3\times ...\times B_r\cong N$
containing $R$ as a subring. That is a contradiction.
\end{proof}

\begin{lemma} \label{existence of Noetherian subrings}
Let $\lambda :R=\mathbf{Z}_{(p)}\rtimes\mathbf{Q}\hookrightarrow B$
be an inclusion of rings, with $B$ indecomposable Noetherian, and
suppose that we have a decomposition $B=\mathbf{Z}_{(p)}\oplus I$,
where  $I$ is an ideal of $B$ containing  $0\rtimes\mathbf{Q}$. Then
$B$ admits a proper Noetherian subring $B'$ containing $R$.
\end{lemma}
\begin{proof}
We claim that $(0\rtimes\mathbf{Q})B+I^2$ is an ideal of $B$
propertly contained in $I$. Indeed the equality
$I=(0\rtimes\mathbf{Q})B+I^2$ would give an epimorphism of
$B$-modules

\begin{center}
$(0\rtimes\mathbf{Q})B\twoheadrightarrow I/I^2$,
\end{center}
thus showing that $I/I^2$ is divisible as a
$\mathbf{Z}_{(p)}$-module. But, on the other and, $I/I^2$ is
finitely generated as a module over the ring
$B/I\cong\mathbf{Z}_{(p)}$. Therefore we would get that $I/I^2=0$
and hence would find an idempotent $e=e^2\in I$ such that $I=Be$.
That would contradict the fact that $B$ is indecomposable.

Since our claim is true we can take the proper subring
$B'=\mathbf{Z}_{(p)}\oplus [(0\rtimes\mathbf{Q})B+I^2]$ of $B$. An
argument already used in the proof of Lemma \ref{decomposition of
the envelope} shows that $B$ is finitely generated as $B'$-module,
and hence that $B'$ is Noetherian.
\end{proof}

\begin{lemma} \label{lifting to a proper decomposition}
Let  $\lambda :R=\mathbf{Z}_{(p)}\rtimes\mathbf{Q}\hookrightarrow B$
be an inclusion of rings, with $B$ indecomposable Noetherian.
Suppose that $\mathbf{m}$ is a maximal ideal of $B$ and that
$g:A\longrightarrow B$ is a homomorphism of Noetherian
$\mathbf{Z}_{(p)}$-algebras such that the composition

\begin{center}
$A\stackrel{g}{\longrightarrow}B\stackrel{pr}{\twoheadrightarrow}B/\mathbf{m}$
\end{center}
is surjective. Then  $\tilde{B}=A\oplus\mathbf{m}$ has a structure
of Noetherian ring, with multiplication
$(a,m)(a',m')=(aa',g(a)m'+mg(a')+mm')$, such that the map $\psi
:\tilde{B}\longrightarrow B$, $(a,m)\rightsquigarrow g(a)+m$, is a
(surjective) ring homomorphism.
\end{lemma}
\begin{proof}
Clearly the  multiplication given on $\tilde{B}$ makes it into a
ring, and the canonical map $\psi :\tilde{B}\twoheadrightarrow B$ is
a surjective ring homomorphism. Its kernel consists of those pair
$(a,m)\in\tilde{B}$ such that $g(a)+m=0$, which gives the equality

\begin{center}
$Ker(\psi )=\{(a,-g(a)):$ $a\in g^{-1}(\mathbf{m})\}$.
\end{center}
Note that every $\tilde{B}$-submodule of $Ker(\psi )$ is canonically
an $A$-submodule and that we have an isomorphism of $A$-modules
$Ker(\psi )\cong g^{-1}(\mathbf{m})$. Since $g^{-1}(\mathbf{m})$ is
an ideal of the Noetherian ring $A$, we conclude that  $Ker (\psi )$
is a Noetherian $\tilde{B}$-module. This and the fact that  the ring
$\tilde{B}/Ker(\psi )\cong B$ is  Noetherian imply that $\tilde{B}$
is a Noetherian ring.
\end{proof}

\begin{lemma} \label{lifting which is a retraction}
Suppose that  in the situation of last lemma, we have
$A=\mathbf{Z}_{(p)}[X]$ and $B/\mathbf{m}\cong\mathbf{Q}$ . If the
homomorphism  $\psi :\tilde{B}\twoheadrightarrow B$,
$(a,m)\rightsquigarrow g(a)+m$,  is a retraction in $CRings$, then
either $B$ contains a proper Noetherian subring containing $R$ or
there is a maximal ideal $\mathbf{m}'$ of $B$ such that
$B=\mathbf{Z}_{(p)}+\mathbf{m}'$.
\end{lemma}
\begin{proof}
We fix a section $\varphi :B\longrightarrow\tilde{B}$ for $\psi$ in
$CRings$. Then we put $\mathbf{q}':=\varphi^{-1}(0\oplus\mathbf{m})$
and $A':=B/\mathbf{q}'$. We get a subring $A'$ of
$\mathbf{Z}_{(p)}[X]$ (whence $A'$ is an integral domain) containing
$\mathbf{Z}_{(p)}$. Moreover, since $p$ is not invertible in
$\mathbf{Z}_{(p)}[X]$ it cannot be invertible in $A'$. Therefore
$pA'\neq A'$ and we have an induced ring homomorphism
$\bar{\varphi}:A'/pA'\longrightarrow\mathbf{Z}_{(p)}[X]/p\mathbf{Z}_{(p)}[X]\cong\mathbf{Z}_p[X]$.
We denote by $C$ its image, which is then a subring of
$\mathbf{Z}_p[X]$ isomorphic to $B/\mathbf{q}$, for some
$\mathbf{q}\in Spec(B)$ such that
$\mathbf{q}'+pB\subseteq\mathbf{q}$. Then  the composition
$\mathbf{Z}_{(p)}\hookrightarrow B\twoheadrightarrow B/\mathbf{q}=C$
has kernel $p\mathbf{Z}_{(p)}$.

We distinguish two situations. In case the last composition is
surjective, and hence $\mathbf{Z}_p\cong C$, we have that
$\mathbf{q}$ is a maximal ideal of $B$ such that
$B=\mathbf{Z}_{(p)}+\mathbf{q}$ and the proof is finished. In case
the mentioned composition is not surjective, there exists a
nonconstant polynomial $f=f(X)\in\mathbf{Z}_p[X]$ such that $f\in
C$. There is no loss of generality in taking $f$ to be monic, so
that $X$ is integral over $\mathbf{Z}_p[f]$ and, hence, the
inclusion $C\subseteq\mathbf{Z}_p[X]$ is an integral extension. In
particular, we have $K-dim(C)=1$ and the assignment
$\mathbf{q}\rightsquigarrow C\cap \mathbf{q}$ gives a surjective map
$Max(\mathbf{Z}_p[X])\twoheadrightarrow Max(C)$ (cf.
\cite{Kunz-85}[Corollary II.2.13]).

If  $\mathbf{n}$ is a maximal ideal of $C$ and we put
$\mathbf{n}=C\cap\hat{\mathbf{n}}$, with $\hat{\mathbf{n}}\in
Max(\mathbf{Z}_p[X])$, then we get a field homomorphism
$C/\mathbf{n}\longrightarrow\mathbf{Z}_p[X]/\hat{\mathbf{n}}$. In
particular,
  $C/\mathbf{n}$ is a finite field
extension of $\mathbf{Z}_p$. Take now $\mathbf{n}'\in Max(B)$ such
that $\mathbf{n}=\mathbf{n}'/\mathbf{q}$. One easily sees that
$B'=\mathbf{Z}_{(p)}+\mathbf{n}'$ is a subring of $B$ such that $B$
is finitely generated as $B'$-module, and then, by
 Eakin's theorem, we know that  $B'$ is Noetherian.
But  $0\rtimes Q$ is contained in all maximal ideals of $B$ since it
consists of ($2$-)nilpotent elements. In particular, we get that
$B'$ contains $R=\mathbf{Z}_{(p)}\rtimes\mathbf{Q}$ and the proof is
finished.
\end{proof}

We are now ready to give the desired proof.

\vspace*{0.3cm}

{\bf Proof of Theorem \ref{nonexistence of Noetherian envelope}:}

Put $R=\mathbf{Z}_{(p)}\rtimes\mathbf{Q}$ as usual and suppose that
it has a Noetherian envelope, represented by a matrix map as in
Lemma \ref{decomposition of the envelope}. We first prove that at
least one of the $B_i$ of Lemma \ref{decomposition of the envelope}
has a maximal ideal $\mathbf{m}$ such that
$B_i/\mathbf{m}\cong\mathbf{Q}$. Indeed, preserving the notation of
the proof of Lemma \ref{decomposition of the envelope}, we see that
the map $\rho :I=B\longrightarrow\mathbf{Q}\rtimes\mathbf{Q}$ is a
surjective ring homomorphism. But, since
$\mathbf{Q}\rtimes\mathbf{Q}$ is indecomposable,  $\rho$ necessarily
vanishes on all but one of the $B_i$ appearing in the decomposition
$B=B_2\times ...\times B_r$. Then we get a unique index $i$ such
that $\rho_{|B_i}:B_i\longrightarrow\mathbf{Q}\rtimes\mathbf{Q}$ is
nonzero, and hence $\rho_{|B_i}$ is surjective. Now
$\mathbf{m}=\rho_{|B_i}^{-1}(0\rtimes\mathbf{Q})$ is a maximal ideal
of $B_i$ such that $B_i/\mathbf{m}\cong\mathbf{Q}$.

Let fix now $i\in\{2,...,r\}$  such that $B_i$ admits a maximal
ideal $\mathbf{m}$ with $B_i/\mathbf{m}\cong\mathbf{Q}$. For
simplification, put $C=B_i$. We fix a surjective ring homomorphism
$\Psi :C\longrightarrow\mathbf{Q}$ with kernel $\mathbf{m}$ and fix
an element $x\in C$ such that $\Psi (x)=p^{-1}$. If $X$ is now a
variable over $\mathbf{Z}_{(p)}$, then the assignment
$X\rightsquigarrow x$ induces a homomorphism of Noetherian
$\mathbf{Z}_{(p)}$-algebras, $g:\mathbf{Z}_{(p)}[X]\longrightarrow
C$ such that the composition

\begin{center}
$\mathbf{Z}_{(p)}[X]\stackrel{g}{\longrightarrow}C\stackrel{pr}{\twoheadrightarrow}C/\mathbf{m}$
\end{center}
is surjective. According to Lemma \ref{lifting to a proper
decomposition}, we know that
$\tilde{C}=\mathbf{Z}_{(p)}[X]\oplus\mathbf{m}$ has a structure of
Noetherian ring such that the canonical map $\psi
:\tilde{C}\longrightarrow C$, $(a,m)\rightsquigarrow g(a)+m$, is a
surjective ring homomorphism. Note that we have an obvious
(injective) ring homomorphism
$h:R=\mathbf{Z}_{(p)}\rtimes\mathbf{Q}\longrightarrow\tilde{C}=\mathbf{Z}_{(p)}[X]\oplus\mathbf{m}$
induced by the inclusions
$\mathbf{Z}_{(p)}\hookrightarrow\mathbf{Z}_{(p)}[X]$ and
$0\rtimes\mathbf{Q}\stackrel{\lambda}{\hookrightarrow}\mathbf{m}$.
Such a ring homomorphism has the property that $\psi
h=\lambda_i:R\longrightarrow B_i=C$.  It is not difficult to see
that the only idempotent elements of $\tilde{C}$ are the trivial
ones, so that $\tilde{C}$ is an indecomposable ring. By Lemma
\ref{decomposition of the envelope}, the morphism $h$ factors
through some $\lambda_j$ ($j=2,...,r$). Fix such an index $j$ and
take then a ring homomorphism $h':B_j\longrightarrow\tilde{C}$ such
that $h'\lambda_j=h$. Then we have that $\psi h'\lambda_j=\psi
h=\lambda_i$. Again by Lemma \ref{decomposition of the envelope}, we
get that $i=j$ and that $\psi h'$ is an isomorphism. In particular,
we get that $\psi$ is a retraction in $CRings$.

Now from Lemmas \ref{lifting which is a retraction} and
\ref{decomposition of the envelope}, we conclude that $C=B_i$ has a
maximal ideal $\mathbf{m}'$ such that
$\mathbf{Z}_{(p)}+\mathbf{m}'=C$.  Then, according to Lemma
\ref{lifting to a proper decomposition},
$\tilde{C}=\mathbf{Z}_{(p)}\oplus\mathbf{m}'$ gets a structure of
Noetherian (indecomposable) ring, with multiplication
$(a,m)(a',m')=aa'+am'+ma'+mm'$, so that the canonical map $\psi
:\tilde{C}\longrightarrow C$, $(a,m)\rightsquigarrow a+m$,  is a
surjective ring homomorphism. An argument similar to the one in the
previous paragraph shows that $\psi$ is a retraction in $CRings$. We
again fix a section for it $\varphi ':C\longrightarrow\tilde{C}$.
Notice that the composition

\begin{center}
$\varphi_1:C\stackrel{\varphi
'}{\longrightarrow}\tilde{C}=\mathbf{Z}_{(p)}\oplus\mathbf{m}'\stackrel{\begin{pmatrix}
1 & 0\end{pmatrix}}{\longrightarrow}\mathbf{Z}_{(p)}$
\end{center}
is a ring homomorphism such that $\varphi
'(b)=(\varphi_1(b),b-\varphi_1(b))$, for all $b\in B$. The universal
property of localization with respect to multiplicative subsets
implies that the only ring endomorphism of $\mathbf{Z}_{(p)}$ is the
identity map, so that
$\varphi_{1_{|\mathbf{Z}_{p}}}=1_{\mathbf{Z}_{(p)}}:\mathbf{Z}_{(p)}\longrightarrow\mathbf{Z}_{(p)}$.
Therefore we get that $\varphi' (a)=(a,0)$, for all
$a\in\mathbf{Z}_{(p)}$. That proves that $\mathbf{Z}_{(p)}\cap
Ker(\varphi ')=0$. But since $\varphi_1$ is surjective we conclude
that we have a $\mathbf{Z}_{(p)}$-module decomposition
$B=\mathbf{Z}_{(p)}\oplus I$, where $I=Ker (\varphi_1)$. On the
other hand, since $(0\rtimes\mathbf{Q})^2=0$ and $\mathbf{Z}_{(p)}$
is an integral domain, we conclude that
$\varphi_1(0\rtimes\mathbf{Q})=0$,  and so that
$0\rtimes\mathbf{Q}\subseteq I$. By Lemma \ref{existence of
Noetherian subrings}, we get that $C=B_i$ contains a proper
Noetherian subring containing $R$. That contradicts Lemma
\ref{decomposition of the envelope} and ends the proof.

We end the paper by proposing:

\begin{conjecture}
\begin{enumerate}
\item There does not exist any non-Noetherian commutative ring having a monomorphic Noetherian envelope
 \item A commutative ring $R$ has a Noetherian envelope if, and
only if, it has a nil ideal $I$ such that $R/I$ is Noetherian and
$\mathbf{p}I_{\mathbf{p}}=I_\mathbf{p}$, for all $\mathbf{p}\in
Spec(R)$.
\end{enumerate}
\end{conjecture}

By Theorem \ref{epimorphic Noetherian envelopes} and  our comments
at the beginning of this section, the two conjectures above are
equivalent.

\end{document}